%% file: ChimeraEntropyRev1.tex
\documentclass{elsarticle}
\usepackage{geometry}                
\usepackage{graphicx}
\usepackage[reqno]{amsmath}
\usepackage{amssymb}
\let\vec\relax 
\usepackage{epstopdf}
\usepackage{amsthm}
\let\vec\relax 
\usepackage[]{graphicx}               
\usepackage{bm}
\usepackage{esvect}
\usepackage{ulem}
\usepackage{adjustbox}
\usepackage{accents}
\usepackage{stmaryrd}
\usepackage{xcolor}
\usepackage{enumitem}
\usepackage{hyperref}
\usepackage{booktabs}
\usepackage{enumitem}
\include{newcommands}

\usepackage{xargs}    
\usepackage[colorinlistoftodos,prependcaption,textsize=tiny]{todonotes}
\newcommandx{\unsure}[2][1=]{\todo[linecolor=blue,backgroundcolor=blue!25,bordercolor=blue,#1]{#2}}
\newcommandx{\changeThis}[2][1=]{\todo[linecolor=red,backgroundcolor=red!25,bordercolor=red,#1]{#2}}

\newtheorem{thm}{Theorem}

\newtheorem{rem}{Remark}

\begin{document}
\begin{frontmatter}
\title{On the Theoretical Foundation of Overset Grid Methods for Hyperbolic Problems II: Entropy Bounded Formulations for Nonlinear Conservation Laws}

\author[1]{David A. Kopriva}
\ead{kopriva@math.fsu.edu}
\address[1]{Department of Mathematics, Florida State University, Tallahassee, FL 32306, USA and Computational Science Research Center, San Diego State University, San Diego, CA 92182, USA.}
\author[2]{Gregor J. Gassner}
\ead{ggassner@uni-koeln.de}
\address[2]{Department for Mathematics and Computer Science; Center for Data and Simulation Science,
 University of Cologne, Weyertal 86-90, 50931, Cologne, Germany.}
\author[3]{Jan  Nordstr\"om\corref{cor3}}
\ead{jan.nordstrom@liu.se}
\address[3]{Department of Mathematics, Applied Mathematics, Linköping University, 581 83 Linköping, Sweden and Department of Mathematics and Applied Mathematics
 University of Johannesburg
 P.O. Box 524, Auckland Park 2006, South Africa.}
 \cortext[cor3]{Corresponding author}
\begin{abstract}
We derive entropy conserving and entropy dissipative overlapping domain formulations for systems of nonlinear hyperbolic equations in conservation form, such as would be approximated by overset mesh methods. The entropy conserving formulation imposes a two-way coupling at the artificial interface boundaries through nonlinear penalty functions that vanish when the solutions coincide. The penalty functions are expressed in terms of entropy conserving fluxes originally introduced for finite volume schemes. \textcolor{black}{In addition to the interface coupling, which is required, entropy dissipation and coupling can optionally be added through the use of linear penalties within the overlap region.}
\end{abstract}

\end{frontmatter}

\section{Introduction}

Overset grid methods \cite{Chan99},\cite{Meakin:1999io} are used to simplify the mesh generation problem for complex geometries by overlapping simpler--to--mesh domains to cover the original more complex domain. An example is shown in Fig. \ref{fig:ChimeraGridExpand}. Overset grid methods trade the complexity of meshing a complex domain for the complexity of coupling the overlapping grids. In addition to the physical boundaries, $\Gamma_a$ and $\Gamma_d$ in Fig.  \ref{fig:ChimeraGridExpand}, artificial interior boundaries, $\Gamma_b$ and $\Gamma_c$, are introduced that will need boundary conditions applied to them. \textcolor{black}{Essentially, two ways to couple the domains are found in the literature:  boundary coupling and volume coupling. Boundary coupling typically includes interpolation of solution or characteristic data from the base to the overset grid along $\Gamma_c$, and similarly from the overset grid to the base grid along $\Gamma_b$ \cite{Bodony:2011cl},\cite{STEGER1987301}. Volume coupling transfers the solution from one grid to another, for example through interpolation, see e.g. \cite{brazell2016overset}}.


Overset grids have been used with all the major approximation methods, finite difference, finite element, finite volume and spectral element  \cite{STEGER1987301},\cite{Hansbo:2003eu}\cite{GALBRAITH201427},\cite{brazell2016overset},\cite{Kauffman:2017uj},\cite{OSTI},\cite{Kopriva:1989},\cite{Merrill:2016qj}. No matter how the problems are approximated, however, the approximations must be consistent with \textcolor{black}{and converge to} the original single domain PDE problem. 

Unlike the original single domain problem, the overset grid methods produce multiple solutions in the overlap regions. Two fundamental questions then arise: 
\begin{enumerate}
\item Is the solution of the overset domain problem the same as the original, single domain problem?
\item Is the overset domain problem well-posed (or bounded, if nonlinear)? 
\end{enumerate}
\textcolor{black}{It is imperative that any approximation, including overset grid approximations, be put on a firm theoretical foundation and approximate a meaningful PDE problem equivalent to the original one.} If the PDE problems being approximated are not well-posed or bounded, then the question of stability of the numerical schemes becomes moot. If they do not have the same solution as the original problem, they are not consistent, and hence meaningless.
\begin{figure}[tbp] 
   \centering
   \includegraphics[width=5.5in]{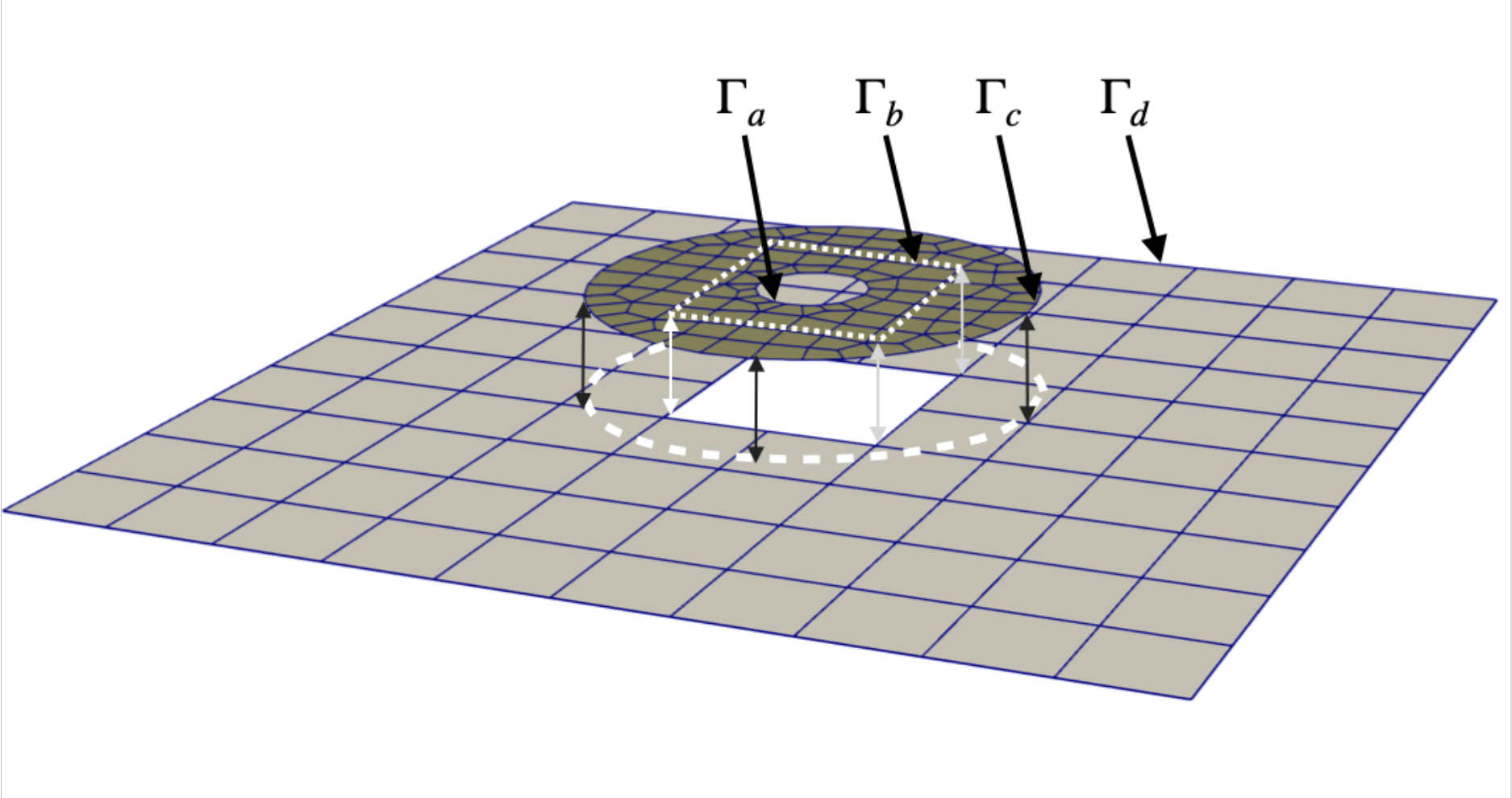} 
   \caption{Exploded view of the overset grid problem showing physical, $\Gamma_a$, $\Gamma_d$, and artificial interior, $\Gamma_b$,$\Gamma_c$, boundaries}
   \label{fig:ChimeraGridExpand}
\end{figure}

In a recent paper, \cite{KOPRIVA2022110732}, hereinafter called Paper I, we studied the overset domain problems for linear hyperbolic systems. We showed that in general, when one cannot characteristically decouple the solutions in the overlap regions, the coupling between the domains must be {\it two-way} for the problem to be well-posed. It is not sufficient to pass only data from the base to the overset domain along $\Gamma_c$ and vice versa along $\Gamma_b$. If limited to one--way coupling, there are terms in the equation for the norm of the solution (energy) for which boundary or initial data is needed, but is not available, and so the problem is not well--posed. 

To impose two--way coupling between the subdomains and to form a well-posed problem, we proposed a novel formulation of the continuous overset domain problem for linear hyperbolic systems. Conditions were applied along $\Gamma_b$ and $\Gamma_c$ on {\it both} domains through penalty terms proportional to the jump in the two solutions. The penalty terms applied to the interior of the subdomain serves to cancel those terms for which no data is available. Under general, but well--defined, conditions on the coupling matrices in the penalty terms, the overset domain problem with interior penalties is energy bounded. Since the problems were linear, energy boundedness could then be used to show that starting with the initial conditions of the single domain problem on each subdomain, the solution of the overset domain problem was identical to the original one, and hence also well-posed. Satisfying another well--defined condition on the penalty matrices ensured conservation. 

We also \textcolor{black}{enabled volume coupling by adding }optional penalty terms to the interior of the overlap region, which could be used to more tightly couple the solutions between the subdomains. \textcolor{black}{The volume coupling was also two-way, and proportional to the difference between the solutions on the two domains, unlike what is traditionally done in numerical implementations of overset grid methods where data is often interpolated from one grid to another. With the penalty formulation, however, the volume coupling was shown to preserve well-posedness and conservation of the overset problem with two-way coupling of the interface penalties, again under well-defined conditions on the coefficient matrices.}

The formulation proposed in Paper I is suitable for linear problems such as Maxwell's equations or linear acoustics, but not for nonlinear problems, such as the Euler gas-dynamics or shallow water equations. In this paper we propose formulations suitable for such nonlinear problems.

For the approximation of nonlinear PDEs it has become popular to substitute $L^2$ stability, like that used in Paper I, with discrete {\it entropy stability}, where a mathematical entropy function can be shown to be bounded by initial and boundary data.  For a review, see \cite{Tadmor:2003fv}. For some systems, like the Burgers equation or shallow water equations, the $L^2$ energy norm is an entropy, so entropy boundedness ensures that the solution components are bounded. In other systems, it is possible to define the entropy so that the solution is bounded. For a discussion on the relation between entropy and $L^2$ stability, see \cite{Dutt:1988}, for example. Entropy preserving and entropy stable schemes have been found to be more robust than their linearly stable counterparts, hence their popularity. To ensure that the approximation is entropy stable and consistent, however, the continuous PDE problem it approximates must be entropy bounded and consistent. 

To the best of our knowledge, an examination of entropy bounded formulations of the PDEs approximated by overset methods has not been previously done.

In this paper we propose entropy bounded formulations of the overset grid problem. Following the recipe of Paper I, we enforce two-way coupling through the addition of penalty functions to the PDEs. Unlike those designed for linear problems, the penalty functions are now nonlinear, and designed to ensure that only physical boundary terms remain in the time derivative of the total entropy. As for the energy in the linear case, we define a global entropy that accounts for the overlap region in a consistent way. We also show that penalty functions that are linear in the entropy variables can be added to the overlap region \textcolor{black}{to add optional volume coupling} to more strongly couple the separate domains, \textcolor{black}{and at the same time preserve entropy boundedness and conservation}.

\section{The Overset Domain Problem}

The goal is to find a solution, $\statevecGreek \omega$, to a nonlinear system of conservation laws,
\begin{equation}
\partial_t\statevecGreek\omega+ \vecNablaX\cdot\bigstatevec{f}(\statevecGreek\omega) = 0, \quad \spacevec x\in\Omega, t>0
\label{eq:ConservationLawOnOmega}
\end{equation}
on a domain $\Omega$, known here as the ``original problem", as sketched in Fig. \ref{fig:2DGemetry}. Here, $\statevecGreek \omega$ is the state vector of $p$ unknowns, and $\bigstatevec f = \sum_{i=1}^d \statevec f_i(\statevecGreek \omega)\hat x_i$ is the nonlinear flux vector, for a $d$ dimensional geometry in $\spacevec x$ with unit direction vectors $\hat x_i,\; i= 1,\ldots,d$. In this paper we use $\partial_\xi$ for $\frac{\partial}{\partial\xi}$ for the derivative with respect to some variable $\xi$ as the shorthand for partial derivatives; subscripts are used as identifiers.
\begin{figure}[tbp] 
   \centering
   \includegraphics[width=2.0in]{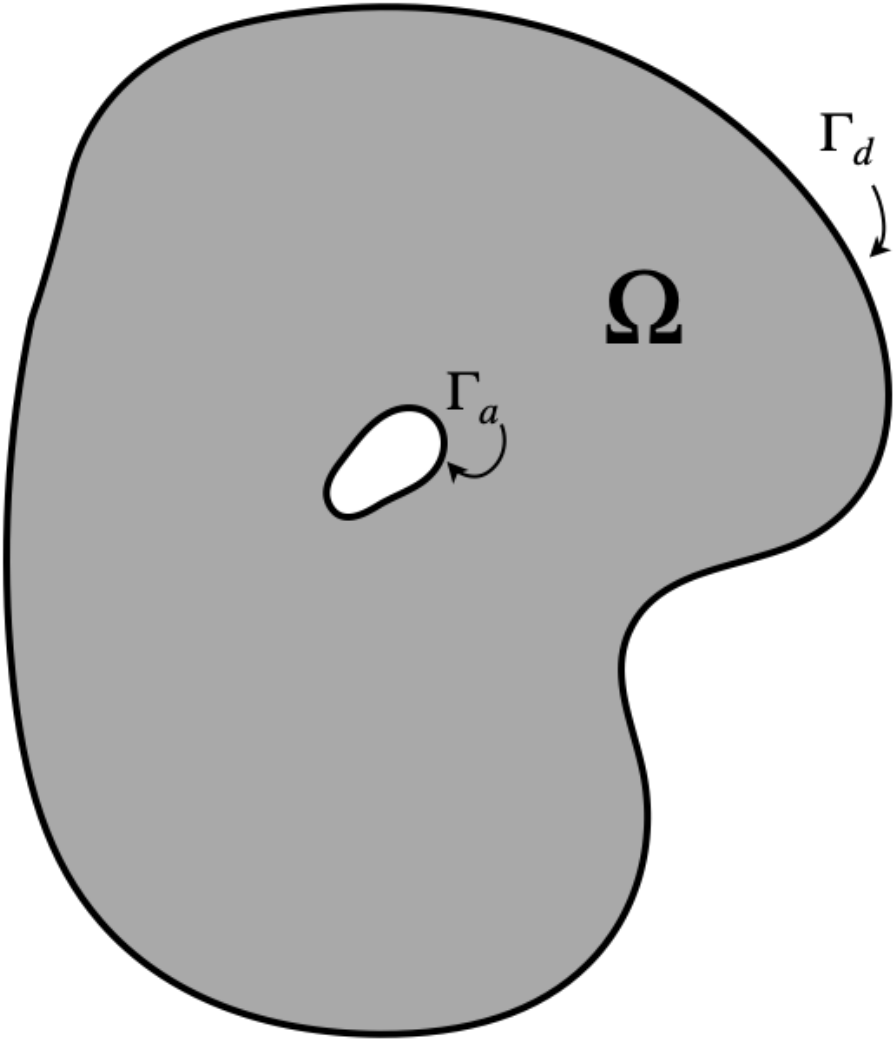} 
   \caption{Diagram of the original problem domain $\Omega$,  sketched in 2D}
   \label{fig:2DGemetry}
\end{figure}

We assume that the nonlinear system \eqref{eq:ConservationLawOnOmega} is entropy bounded. \textcolor{black}{Thus, see e.g. \cite{Tadmor:2003fv}, }we assume that there is an entropy pair $(s,\spacevec f^\epsilon)$, where $s$ is the convex scalar entropy function and $\spacevec f^\epsilon$ is the entropy flux space--vector. 
Associated with the entropy pair are entropy variables, $\statevec w = \partial s/\partial \statevecGreek\omega$, which contract with the state variables as
\begin{equation}
\begin{gathered}
\statevec w^T\partial_t\statevecGreek\omega = \left( \frac{\partial s}{\partial\statevecGreek\omega}\right)^T\partial_t\statevecGreek\omega = \partial_t s(\statevecGreek\omega)\hfill\\
\statevec w^T\vecNablaX \cdot\bigstatevec f = \vecNablaX\cdot\spacevec f^\epsilon,\hfill
\end{gathered}
\label{eq:GeneralContraction}
\end{equation}
so that, for smooth solutions, the entropy satisfies the scalar conservation law
\begin{equation}
\partial_t s + \vecNablaX\cdot\spacevec f^\epsilon = 0, \quad \spacevec x\in\Omega,
\label{eq:GeneralEntropyEquation}
\end{equation}
with equality, $=$, replaced by the inequality, $\le$, for non-smooth solutions. 

One can show that the entropy is bounded by initial and boundary data from \eqref{eq:GeneralEntropyEquation}. Define the inner product
\begin{equation}
\iprod{\statevec u,\statevec v}_\Omega = \int_\Omega \statevec u^T\statevec v d\spacevec x
\end{equation}
for any state vectors $\statevec u,\statevec v$, and similarly for scalars. Then
\begin{equation}
\iprod{\statevec w, \partial_t\statevecGreek\omega}_\Omega + \iprod{\statevec w,  \vecNablaX\cdot\bigstatevec{f}}_\Omega = \iprod{\partial_t s,1}_\Omega + \iprod{\vecNablaX\cdot\spacevec f^\epsilon,1}_\Omega=0.
\label{eq:ContractionOfSystem}
\end{equation}
If we define the total entropy on $\Omega$ as $\bar s_\Omega = \iprod{s,1}_\Omega  = \int_\Omega s d\Omega$, and use Gauss' law on \eqref{eq:ContractionOfSystem}, the total entropy satisfies
\begin{equation}
\frac{d\bar s_\Omega}{dt} + \int_{\partial \Omega} \spacevec f^\epsilon\cdot\hat n \dS = 0.
\label{eqEntropyOverOmega}
\end{equation}
In this way, the total entropy is determined solely by physical boundary and initial values. In this paper we will assume that physical boundary conditions are applied so that
\begin{equation}
 \int_{\partial \Omega} \spacevec f^\epsilon\cdot\hat n \dS  \ge 0,
\end{equation}
and then that initial conditions are specified so that 
\begin{equation}
\bar s_\Omega(t)\le \bar s_\Omega(0)
\end{equation}
for all $t\ge 0$ \cite{Dutt:1988}.

The overset domain problem subdivides the domain $\Omega$ into overlapping subdomains, $\Omega_k$ that are geometrically simpler, and completely cover the domain, $\Omega = \bigcup_k\Omega_k$. In this paper we will limit the problem to two domains,  $\Omega_u$, $\Omega_v$, as sketched in Fig. \ref{fig:2DOversetGemetry} in two space dimensions. The base domain, $\Omega_v$, has an artificial hole bounded by the curve $\Gamma_b$. The overset domain extends beyond that hole to an artificial outer boundary bounded by the curve, $\Gamma_c$. 
At a minimum, boundary conditions are required along the physical boundaries $\Gamma_a$, $\Gamma_d$, but also at the artificial subdomain boundaries, $\Gamma_b$ and $\Gamma_c$.

\begin{figure}[tbp] 
   \centering
   \includegraphics[width=5.5in]{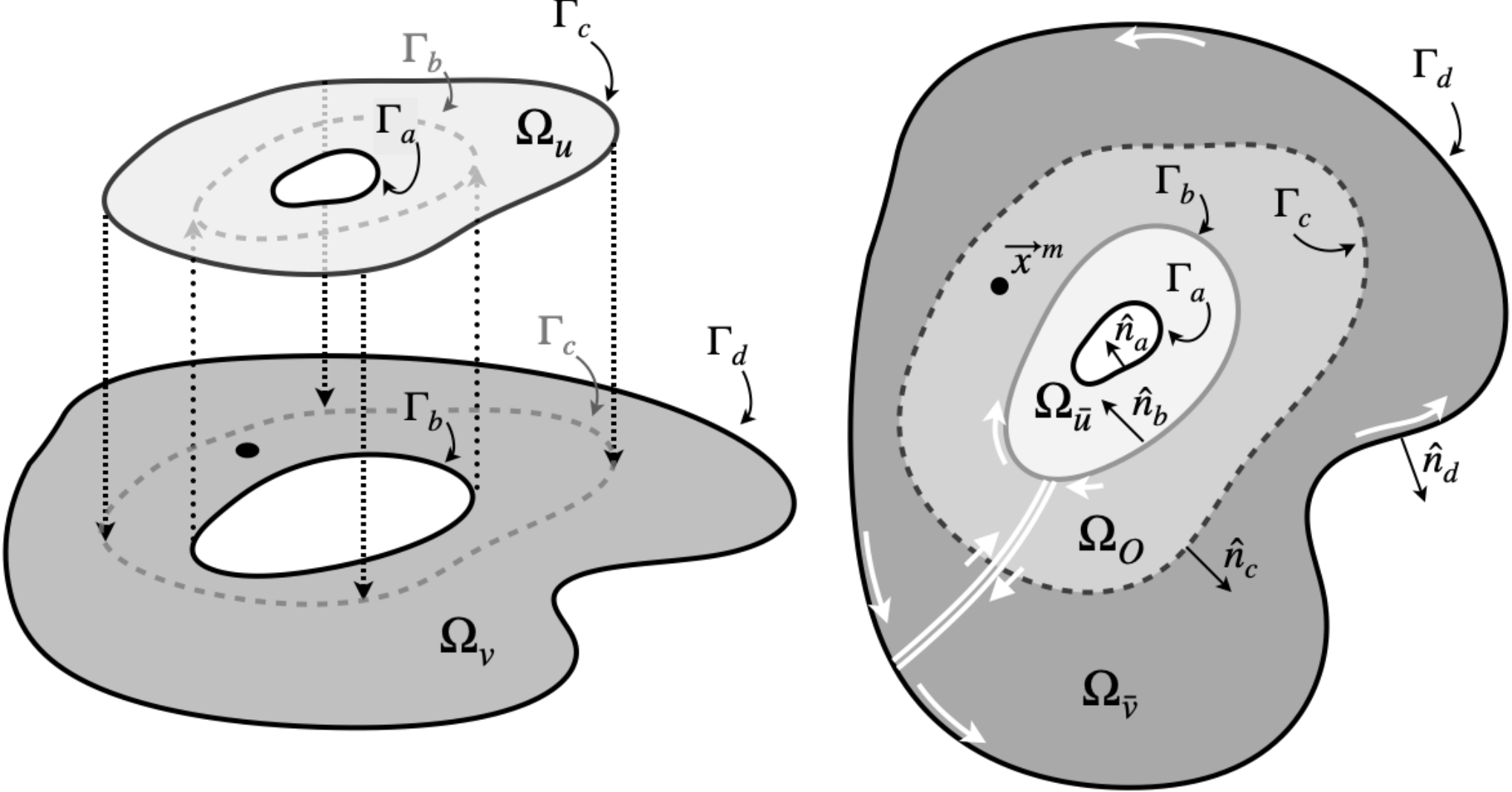} 
   \caption{Diagrams of the overset geometry in 2D \textcolor{black}{with the vertical projection of the domains and curves on the right. The physical boundaries are $\Gamma_a$ and $\Gamma_d$. The interior interfaces are $\Gamma_b$ and $\Gamma_b$. The normals are defined to point to the exterior of the overset domains, $\Omega_u$ and $\Omega_v$. Boundaries are traversed in a counter-clockwise fashion as illustrated by the white arrows and paths on the right. As a result of their definition in terms of the original overset domains, the normals on the complementary domains $\Omega_{\bar u}$ and $\Omega_{\bar v}$ are not in standard outward form. (C.f. $\hat n_c$ pointing into subdomain $\Omega_{\bar v}$.)}}
   \label{fig:2DOversetGemetry}
\end{figure}

In Paper I, we showed that for linear problems it was not sufficient to only apply dissipative boundary conditions along the physical and subdomain boundaries. To guarantee bounded energy for linear problems we introduced linear penalty terms  \cite{Nordstrom:2014qd,NORDSTROM2021109821} to the conservation laws in each domain to enforce two--way coupling. 

Following the approach of Paper I, we now write the overset (O) and base (B) PDEs with nonlinear penalty terms applied along the interface boundaries. To do so, let
the function $\mathcal L^\alpha$  be a lifting operator  \cite{Arnoldetal2002} that returns the values along a curve, $\Gamma_\alpha$. For example, over a two-dimensional domain $V$ containing the curve $\Gamma_\alpha$,
\begin{equation}
\int_V \statevecGreek \psi^T\mathcal L^\alpha\left[\statevec q\right]dV = \int_{\Gamma_\alpha} \statevecGreek \psi^T(\spacevec x(l))\statevec q(\spacevec x(l))dl,
\label{eq:2DLifting}
\end{equation}
where $\statevecGreek \psi$ is some test function. The lifting function applied to some function, $\statevec q$, will just become the function $\statevec q$ evaluated along the curve in the weak form of the equations. 

Let $\mathcal P^\alpha_{u}(\statevec u, \statevec v)$, $\mathcal P^\alpha_{v}(\statevec u, \statevec v)$ be penalty functions defined along a curve $\Gamma^\alpha$ with the property that for any admissible $\statevecGreek \omega$, $\mathcal P_{\bullet}(\statevecGreek \omega, \statevecGreek \omega)=0$, so that the penalty vanishes when there is no difference in the arguments. \textcolor{black}{This condition is required for consistency with the original equations so that the penalty terms vanish when $u=v$.} Later, we will show that, in general, $\mathcal P_{\bullet}^\alpha = \vec{\mathcal P}_{\bullet}\cdot \hat n_\alpha$ for the unit normal $\hat n_\alpha$ along $\Gamma_\alpha$.

In strong form, the overset domain equations with two-way coupling are
\begin{equation}
\begin{gathered}
(O)\quad \partial_t \statevec u + \vecNabla_x \cdot \bigstatevec f(\statevec u) + \mathcal L^b\left[\mathcal P^b_{u}(\statevec u, \statevec v)\right] + \mathcal L^c\left[\mathcal P^c_{u}(\statevec u, \statevec v)\right]= 0\quad \vec x\in \Omega_u\hfill\\
(B)\quad \partial_t \statevec v + \vecNabla_x \cdot \bigstatevec f(\statevec v) + \mathcal L^b\left[\mathcal P^b_{v}(\statevec u, \statevec v)\right] + \mathcal L^c\left[\mathcal P^c_{v}(\statevec u, \statevec v)\right]= 0\quad \vec x\in \Omega_v.\hfill
\end{gathered}
\label{eq:TwoEquationStrongForm}
\end{equation}
To specify the full overset domain problem, we assume that dissipative boundary conditions are applied along $\Gamma_a$ and $\Gamma_d$, and proper initial conditions are set; we restrict ourselves in this paper, however, to the study of the effects of the artificial interface boundaries, $\Gamma_b$ and $\Gamma_c$. 

We get the weak form of the overset domain equations by multiplying \eqref{eq:TwoEquationStrongForm} with test functions and integrating over the domains. In inner product notation, the weak form becomes
\begin{equation}
\begin{gathered}
(O)\quad \iprod{\statevecGreek\phi_u,\partial_t \statevec u} +\iprod{\statevecGreek\phi_u,\vecNabla_x \cdot \bigstatevec f(\statevec u)} 
+ \int_{\Gamma_b}\statevecGreek\phi_u^T\mathcal P^b_{u}(\statevec u, \statevec v)\dS
+ \int_{\Gamma_c}\statevecGreek\phi_u^T\mathcal P^c_{u}(\statevec u, \statevec v)\dS= 0\quad \vec x\in \Omega_u\hfill\\
(B)\quad  \iprod{\statevecGreek\phi_v,\partial_t \statevec v} + \iprod{\statevecGreek\phi_v,\vecNabla_x \cdot \bigstatevec f(\statevec v)} + \int_{\Gamma_b}\statevecGreek\phi_v^T\mathcal P^b_{v}(\statevec u, \statevec v)\dS + \int_{\Gamma_c}\statevecGreek\phi_v^T\mathcal P^c_{v}(\statevec u, \statevec v)\dS= 0\quad \vec x\in \Omega_v.\hfill
\end{gathered}
\label{eq:TwoEquationWeakForm}
\end{equation}

\section{Entropy Boundedness and Conservation for the Overset Domain Problem}

As it was necessary to consistently define the total energy for linear problems in Paper I, it is also necessary to define the total entropy for the overset problem here. The issue is that in the overlap region, $\Omega_O$, the entropy gets counted twice if we simply add the entropies of the two domains.

The problem of enforcing entropy boundedness reduces to finding the penalty functions $\mathcal P$ so that the total entropy is determined solely by (physical) boundary and initial data, which makes it necessary to define the total entropy in the overlap region to be consistent with the original problem.
In the linear analysis of Paper I, we used the fact that (see Fig. \ref{fig:2DOversetGemetry})
\begin{equation}
\begin{gathered}
\inorm{\statevec u}^2_{\Omega_u} = \inorm{\statevec u}^2_{\Omega_{\bar u}} + \inorm{\statevec u}^2_{\Omega_O},\\
\inorm{\statevec v}^2_{\Omega_v} = \inorm{\statevec v}^2_{\Omega_{\bar v}} + \inorm{\statevec v}^2_{\Omega_O}.\\
\end{gathered}
\end{equation}
So it was true that, for $0< \eta< 1$, the combinations
\begin{equation}
\begin{gathered}
\inorm{\statevec u}^2_{\Omega_u} - \eta \ \inorm{\statevec u}^2_{\Omega_O} = \inorm{\statevec u}^2_{\Omega_{\bar u}} + (1-\eta)\inorm{\statevec u}^2_{\Omega_O}\ge 0,\\
\inorm{\statevec v}^2_{\Omega_v} - (1-\eta)  \inorm{\statevec v}^2_{\Omega_O} = \inorm{\statevec v}^2_{\Omega_{\bar v}} + \eta \inorm{\statevec v}^2_{\Omega_O}\ge 0,\\
\end{gathered}
\label{eq:NewNorms1DW}
\end{equation}
and 
\begin{equation}
\inorm{\statevec u}^2_{\Omega_u} + \inorm{\statevec v}^2_{\Omega_v} -\left\{\eta \ \inorm{\statevec u}^2_{\Omega_O} + (1-\eta)  \inorm{\statevec v}^2_{\Omega_O}\right\} \ge 0
\label{eq:OversetNorm}
\end{equation}
define norms, being positive unless the argument is zero. The overset norm \eqref{eq:OversetNorm} is consistent with the norm over $\Omega$, since if the solutions match the solution $\statevecGreek\omega$,
\begin{equation}
\inorm{\statevecGreek\omega}^2_{\Omega_u} + \inorm{\statevecGreek\omega}^2_{\Omega_v} -\left\{\eta \ \inorm{\statevecGreek\omega}^2_{\Omega_O} + (1-\eta)  \inorm{\statevecGreek\omega}^2_{\Omega_O}\right\} = \inorm{\statevecGreek\omega}_\Omega^2.
\end{equation}

\begin{rem}
\textcolor{black}{At the PDE level the choice of $\eta$ is arbitrary, but in a numerical implementation it gives one the flexibility to choose $\eta$ to favor one domain over another, say if one grid is better resolved and presumably more accurate.}
\end{rem}
For the entropy analysis, we must combine the component entropies so that the {\it combination is also an entropy}, and contracts with the time derivative and divergence to give the entropy and entropy flux, while being consistent with total entropy over $\Omega$. To that end, let us define a new state vector, $\statevec U$, and new entropy variables, $\statevec W$, to be
\begin{equation}
\statevec U = \left\{ 
\begin{gathered}
\left[\begin{array}{c}\statevec u \\\statevec 0\end{array}\right]\quad \spacevec x\in\Omega_{\bar u}\\
\left[\begin{array}{c}\statevec u \\\statevec v\end{array}\right]\quad \spacevec x\in\Omega_{O}\\
\left[\begin{array}{c}\statevec 0 \\\statevec v\end{array}\right]\quad \spacevec x\in\Omega_{\bar v}
\end{gathered}
\right.
\quad
\statevec W = \left\{ 
\begin{gathered}
\left[\begin{array}{c}\statevec w_u \\\statevec 0\end{array}\right]\quad \spacevec x\in\Omega_{\bar u}\hfill\\
\left[\begin{array}{c}(1-\eta)\statevec w_u \\\eta\statevec w_v\end{array}\right]\quad \spacevec x\in\Omega_{O}\hfill\\
\left[\begin{array}{c}\statevec 0 \\\statevec w_v\end{array}\right]\quad \spacevec x\in\Omega_{\bar v}\hfill
\end{gathered}
\right.
\label{eq:UandWVectors}
\end{equation}
where, to simplify notation,  $\statevec w_u \equiv \statevec w(\statevec u)$, etc. The constant $0<\eta<1$ is a weighting parameter that convexly weights the contributions in the overlap region.
Then
\begin{equation}
\statevec W^T\partial_t \statevec U =  \left\{ 
\begin{gathered}
\partial_t s(\statevec u)\quad \spacevec x\in\Omega_{\bar u}\hfill\\
(1-\eta)\partial_t s(\statevec u) + \eta \partial_t s(\statevec v)
\quad \spacevec x\in\Omega_{O}\hfill\\
\partial_t s(\statevec v)\quad \spacevec x\in\Omega_{\bar v}\hfill
\end{gathered}
\right.
= \frac{\partial}{\partial t} \left\{ 
\begin{gathered}
s(\statevec u)\quad \spacevec x\in\Omega_{\bar u}\hfill\\
(1-\eta)s(\statevec u) + \eta s(\statevec v)
\quad \spacevec x\in\Omega_{O}\hfill\\
s(\statevec v)\quad \spacevec x\in\Omega_{\bar v}\hfill
\end{gathered}
\right. \quad\equiv \frac{\partial S}{\partial t}.
\end{equation}
Note that the entropy, $S$, defined as
\begin{equation}
S(\statevec u, \statevec v) = \left\{ 
\begin{gathered}
s(\statevec u)\quad \spacevec x\in\Omega_{\bar u}\hfill\\
(1-\eta)s(\statevec u) + \eta s(\statevec v)
\quad \spacevec x\in\Omega_{O}\hfill\\
s(\statevec v)\quad \spacevec x\in\Omega_{\bar v}\hfill
\end{gathered}
\right.
\label{eq:OversetEntropyDef}
\end{equation}
 is equal to the entropy $s(\statevecGreek \omega)$ on $\Omega$ when $\statevec u = \statevecGreek \omega$ and $\statevec v = \statevecGreek\omega$. So it is consistent. Furthermore, it is convex on all subdomains since $s(\statevec u)$ and $s(\statevec v)$ are convex, and the combination on $\Omega_O$ is also convex.

Not only do the entropy variables contract with the time derivative of the state vector to give the time derivative of the entropy, it is also true that $\statevec W = \frac{\partial S}{\partial\statevec U}$, for differentiating with respect to the components of $\statevec U$,
\begin{equation}
 \frac{\partial S}{\partial \statevec U} =  \left\{ 
\begin{gathered}
\left[\begin{array}{c} \frac{\partial s}{\partial \statevec u} \\\statevec 0\end{array}\right]\quad \spacevec x\in\Omega_{\bar u}\hfill\\
\left[\begin{array}{c}(1-\eta)\frac{\partial s}{\partial \statevec u} \\\eta\frac{\partial s}{\partial \statevec v}\end{array}\right]\quad \spacevec x\in\Omega_{O}\hfill\\
\left[\begin{array}{c}\statevec 0 \\\frac{\partial s}{\partial \statevec v}\end{array}\right]\quad \spacevec x\in\Omega_{\bar v}\hfill
\end{gathered}
\right.
\quad = \statevec W,
\end{equation}
when we define the derivative of zero with respect to zero to be zero.

 For the flux, let the $i^{th}$ flux component of $\bigstatevec F = \sum_i \statevec F_i\hat x_i$ be organized similarly as
 \begin{equation}
 \statevec F_i = \left\{ 
\begin{gathered}
\left[\begin{array}{c}\statevec f_i(\statevec u) \\\statevec 0\end{array}\right]\quad \spacevec x\in\Omega_{\bar u}\\
\left[\begin{array}{c}\statevec f_i(\statevec u) \\\statevec f_i(\statevec v)\end{array}\right]\quad \spacevec x\in\Omega_{O}\\
\left[\begin{array}{c}\statevec 0 \\\statevec f_i(\statevec v)\end{array}\right]\quad \spacevec x\in\Omega_{\bar v}
\end{gathered}
\right. .
\label{eq:BigFVector}
 \end{equation}
 Then
 \small
 \begin{equation}
\statevec W^T\vecNablaX\cdot\bigstatevec F =  \left\{ 
\begin{gathered}
\vecNabla_x \cdot\spacevec f^\epsilon(\statevec u)\quad \spacevec x\in\Omega_{\bar u}\hfill\\
(1-\eta)\vecNabla_x \cdot\spacevec f^\epsilon(\statevec u) + \eta\vecNabla_x \cdot\spacevec f^\epsilon(\statevec v)
\quad \spacevec x\in\Omega_{O}\hfill\\
\vecNabla_x \cdot\spacevec f^\epsilon(\statevec v)\quad \spacevec x\in\Omega_{\bar v}\hfill
\end{gathered}
\right.
= \vecNabla_x\cdot \left\{ 
\begin{gathered}
\spacevec f^\epsilon(\statevec u)\quad \spacevec x\in\Omega_{\bar u}\hfill\\
(1-\eta)\spacevec f^\epsilon(\statevec u) + \eta \spacevec f^\epsilon(\statevec v)
\quad \spacevec x\in\Omega_{O}\hfill\\
\spacevec f^\epsilon(\statevec v)\quad \spacevec x\in\Omega_{\bar v}\hfill
\end{gathered}
\right. \quad\equiv \vecNabla_x\cdot \spacevec F^\epsilon.
\end{equation}
\normalsize
The derivatives must be interpreted in a weak sense.

Finally, let us write the overset penalty terms as
\begin{equation}
\statevec P^\alpha\left(\statevec u,\statevec v\right) = 
 \left\{ 
\begin{gathered}
\left[\begin{array}{c}\ \mathcal{P}^\alpha_u(\statevec u,\statevec v)  \\\statevec 0\end{array}\right]\quad \spacevec x\in\Omega_{\bar u}\\
\left[\begin{array}{c}\statevec 0 \\\statevec 0\end{array}\right]\quad \spacevec x\in\Omega_{O}\\
\left[\begin{array}{c}\statevec 0 \\  \mathcal{P}^\alpha_v(\statevec u,\statevec v) \end{array}\right]\quad \spacevec x\in\Omega_{\bar v},
\end{gathered}
\right.\
\label{eq:bigPVector}
\end{equation}
where $\alpha$ is either $b$ or $c$. Note that the penalty is not applied to the overlap region, see Paper I \cite{KOPRIVA2022110732} for a discussion.

With the definitions \eqref{eq:UandWVectors},\eqref{eq:BigFVector}, and \eqref{eq:bigPVector}, the overset problem can be written in strong form as
\begin{equation}
\partial_t \statevec U + \vecNabla_x\cdot\bigstatevec F  + \mathcal L^b\left[\statevec P^b\right]+ \mathcal L^c\left[\statevec P^c\right]= 0.
\label{eq:BigOversetEquation}
\end{equation}
Contracting with the entropy variable vector leads to a scalar equation for the entropy of the overset domain problem,
\begin{equation}
\textcolor{black}{
\statevec W^T\partial_t \statevec U + \statevec W^T\vecNabla_x\cdot \bigstatevec F +  \statevec W^T\statevec P= \partial_t S + \vecNabla_x\cdot \spacevec F^\epsilon + P^b + P^c= 0,}
\label{eq:BigEntropyEqn}
\end{equation}
where $P^\alpha=\statevec W^T \mathcal L^\alpha\left[\statevec P^\alpha\right] $, and $S$ is defined in \eqref{eq:OversetEntropyDef}. Thus, the overset domain entropy satisfies
\begin{equation}
\textcolor{black}{
\partial_t S + \vecNabla_x\cdot \spacevec F^\epsilon\le 0}
\label{eq:FullEntropyEquation}
\end{equation}
 as long as the penalty terms are non-negative, with equality when $\statevec u = \statevec v$, as in the original problem, \eqref{eq:GeneralEntropyEquation}.

\begin{rem}
\textcolor{black}{We see, then, that the overset domain entropy variables, $\statevec W$, contract with the overset domain state, $\statevec U$, and the overset domain flux, $\bigstatevec F$, in the same way as in the original single domain problem, ensuring that $S$ is an entropy with an entropy flux $\spacevec F^\epsilon$. 
Thus, the overset domain problem has an entropy pair $(S, \spacevec F^\epsilon)$, just as the original problem has the pair $(s, \spacevec f^\epsilon)$ over the same domain. Furthermore, $S$, defined in \eqref{eq:OversetEntropyDef}, satisfies the same equation as the entropy, \eqref{eq:GeneralEntropyEquation}, with equality when $\statevec u=\statevec v$. 
The overset domain is therefore consistent with the original problem, since when $\statevecGreek\omega = \statevec u = \statevec v$, then $S=s$, on all domains.
This result is of fundamental importance because it shows that the overset domain problem has the same properties as the original one, and when the initial and boundary data match, the overset domain problem has the same solution as the original one. }
\end{rem}

We can re-write the weak form, \eqref{eq:TwoEquationWeakForm}, as a single overset domain problem if we define a test function
\begin{equation}
\statevecGreek\Phi = \left\{ 
\begin{gathered}
\left[\begin{array}{c}\statevecGreek\phi_u \\\statevec 0\end{array}\right]\quad \spacevec x\in\Omega_{\bar u}\hfill\\
\left[\begin{array}{c}(1-\eta)\statevecGreek\phi_u \\\eta\statevecGreek\phi_v \end{array}\right]\quad \spacevec x\in\Omega_{O}\hfill\\
\left[\begin{array}{c}\statevec 0 \\\statevecGreek\phi_v\end{array}\right]\quad \spacevec x\in\Omega_{\bar v}\hfill
\end{gathered}
\right..
\end{equation}
Multiplying \eqref{eq:BigOversetEquation} by $\statevecGreek\Phi$ and integrating over the domain gives the first weak form,
\begin{equation}
\iprod{\statevecGreek\Phi,\partial_t \statevec U}_\Omega + \iprod{\statevecGreek\Phi,\vecNabla_x\cdot\bigstatevec F}_\Omega + \int_{\Gamma_b}\statevecGreek\Phi^T\statevec P^b\dS +  \int_{\Gamma_c}\statevecGreek\Phi^T\statevec P^c \dS = 0.
\label{eq:BigMultiDWeakForm1}
\end{equation}

We get an alternate weak form by applying Gauss' law (multidimensional integration by parts) to the flux integral in \eqref{eq:BigMultiDWeakForm1}. Separating the integrals  \textcolor{black}{and applying the Gauss Law (integration by parts) to each of the volume integrals} (see Fig. \ref{fig:2DOversetGemetry} \textcolor{black}{, noting how the normals flip direction depending on the subdomain)}, and using the usual notation that $\Gamma^\pm$ represents the limit from each side of the curve with respect to the normal \textcolor{black}{to select the correct subdomain value of the global fluxes in \eqref{eq:BigFVector}},
\begin{equation}
\begin{split}
\iprod{\statevecGreek\Phi,\vecNabla_x\cdot\bigstatevec F}_\Omega &= \iprod{\statevecGreek\Phi,\vecNabla_x\cdot\bigstatevec F}_{\Omega_{\bar u}} + \iprod{\statevecGreek\Phi,\vecNabla_x\cdot\bigstatevec F}_{\Omega_{O}} + \iprod{\statevecGreek\Phi,\vecNabla_x\cdot\bigstatevec F}_{\Omega_{\bar v}}
\\&= 
\int_{\Gamma_a}\statevecGreek\Phi^T\bigstatevec F\cdot\hat n_a - \int_{\Gamma_b^+}\statevecGreek\Phi^T\bigstatevec F\cdot\hat n_b -\iprod{\vecNabla_x\statevecGreek\Phi,\bigstatevec F}_{\Omega_{\bar u}}
\\& + \int_{\Gamma_b^-}\statevecGreek\Phi^T\bigstatevec F\cdot\hat n_b
+ \int_{\Gamma_c^-}\statevecGreek\Phi^T\bigstatevec F\cdot\hat n_c - \iprod{\vecNabla_x\statevecGreek\Phi,\bigstatevec F}_{\Omega_{O}}
\\&+ \int_{\Gamma_d}\statevecGreek\Phi^T\bigstatevec F\cdot\hat n_d - \int_{\Gamma_c^+}\statevecGreek\Phi^T\bigstatevec F\cdot\hat n_c
 - \iprod{\vecNabla_x\statevecGreek\Phi,\bigstatevec F}_{\Omega_{\bar v}}.
\end{split}
\label{eq:FluxIntegralExpansion}
\end{equation}

\subsection{Conditions for Entropy Boundedness}\label{EntropyBoundednessSection}

To get the entropy bound, we use \eqref{eq:BigMultiDWeakForm1} with $\statevecGreek\Phi$ replaced with $\statevec W$,
\begin{equation}
\iprod{\statevec W,\partial_t \statevec U }_\Omega + \iprod{\statevec W,\vecNabla_x\cdot\bigstatevec F }_\Omega 
+ \int_{\Gamma_b}\statevec W^T\statevec P^b\dS +  \int_{\Gamma_c}\statevec W^T\statevec P^c \dS= 0.
\end{equation}
and contract the state and flux terms. Then with
$
\textcolor{black}{\frac{d}{dt}\bar S \equiv \iprod{\partial_t S,1}_\Omega,}
$
\begin{equation}
\textcolor{black}{\frac{d}{dt}\bar S + \iprod{\vecNabla_x\cdot \spacevec F^\epsilon,1}_\Omega +    \int_{\Gamma_b}\statevec W^T\statevec P^b\dS +  \int_{\Gamma_c}\statevec W^T\statevec P^c \dS= 0.}
\label{eq:totalEntropy1}
\end{equation}

We separate boundary and interior contributions when we apply Gauss' law to the divergence of the flux integral.  Note that (see Fig. \ref{fig:2DOversetGemetry} and \eqref{eq:FluxIntegralExpansion})
\begin{equation}
\begin{split}
\textcolor{black}{\iprod{\vecNabla_x\cdot \spacevec F^\epsilon,1}_\Omega} &= \int_{\Omega_{\bar u}}\vecNabla_x\cdot \spacevec f^\epsilon d\Omega_{\bar u}+  \int_O\vecNabla_x\cdot \spacevec f^\epsilon d\Omega_O +  \int_{\Omega_{\bar v}}\vecNabla_x\cdot \spacevec f^\epsilon d\Omega_{\bar v}.
\\& = \left\{\int_{\Gamma_a}  \spacevec f^\epsilon(\statevec u) \cdot\hat n_a\dS - \int_{\Gamma_b}  \spacevec f^\epsilon(\statevec u) \cdot\hat n_b\dS \right\} 
\\&  +\left\{\int_{\Gamma_b} \left((1-\eta)\spacevec f^\epsilon(\statevec u) + \eta \spacevec f^\epsilon(\statevec v)\right) \cdot\hat n_b\dS + \int_{\Gamma_c} \left((1-\eta)\spacevec f^\epsilon(\statevec u) + \eta \spacevec f^\epsilon(\statevec v)\right) \cdot\hat n_c\dS \right\} 
 \\& +\left\{\int_{\Gamma_d} \spacevec f^\epsilon(\statevec v) \cdot\hat n_d\dS - \int_{\Gamma_c} \spacevec f^\epsilon(\statevec v) \cdot\hat n_c\dS \right\} 
 \\& = \left\{\int_{\Gamma_a}  \spacevec f^\epsilon(\statevec u) \cdot\hat n_a\dS  + \int_{\Gamma_d} \spacevec f^\epsilon(\statevec v) \cdot\hat n_d\dS \right\}
 \\&+  \int_{\Gamma_b} \eta\left( \spacevec f^\epsilon(\statevec v)  -  \spacevec f^\epsilon(\statevec u)\right)\cdot\hat n_b\dS
 + \int_{\Gamma_c} (1-\eta)\left( \spacevec f^\epsilon(\statevec u)  -  \spacevec f^\epsilon(\statevec v)\right)\cdot\hat n_c\dS.
\end{split}
\label{eq:globalentropyfluxintegral}
\end{equation} 

Substituting \eqref{eq:globalentropyfluxintegral} into \eqref{eq:totalEntropy1} and gathering terms,
\begin{equation}
\begin{split}
\frac{d}{dt}\bar S &+ \left\{\int_{\Gamma_a}  \spacevec f^\epsilon(\statevec u) \cdot\hat n_a\dS  + \int_{\Gamma_d} \spacevec f^\epsilon(\statevec v) \cdot\hat n_d\dS \right\} 
\\&+ \int_{\Gamma_b} \left[ -\eta\left( \spacevec f^\epsilon(\statevec u)  -  \spacevec f^\epsilon(\statevec v)\right)\cdot\hat n_b + \statevec w_u^T\mathcal P_u^b+ \statevec w_v^T\mathcal P_v^b\right]\dS 
\\&+ \int_{\Gamma_c}\left[(1-\eta)\left( \spacevec f^\epsilon(\statevec u)  -  \spacevec f^\epsilon(\statevec v)\right)\cdot\hat n_c+\statevec w_u^T\mathcal P_u^c + \statevec w_v^T\mathcal P_v^c\right]\dS
\\&= 0.
\end{split}
\label{eq:totalEntropy2}
\end{equation}
The integrands along the interior interfaces $\Gamma_b$ and $\Gamma_c$ are of the same form and can be written generically as
\begin{equation}
\mathcal B \equiv\left( \spacevec f^\epsilon(\statevec u)  -  \spacevec f^\epsilon(\statevec v)\right)\cdot\spacevec \beta+\statevec w_u^T\mathcal P_u + \statevec w_v^T\mathcal P_v,
\label{eq:MultiDEntropyCondition}
\end{equation}
where
\begin{equation}
\spacevec \beta = \left\{
\begin{gathered}
-\eta\hat n_b, \quad \spacevec x\in\Gamma_b \hfill\\ 
(1-\eta)\hat n_c,\quad \spacevec x \in \Gamma_c.
\end{gathered}
   \right.
\label{eq:BetaCondition2D}
\end{equation}

 We then have the entropy boundedness theorem
\begin{thm}
The total entropy, $\bar S$, is bounded by initial and physical boundary data if $\mathcal B\ge 0$ for each boundary curve, $\Gamma_b$, $\Gamma_c$, where $\mathcal B$ is defined in \eqref{eq:MultiDEntropyCondition} and $\spacevec \beta $ is given by \eqref{eq:BetaCondition2D}.
We define the overset domain problem as entropy preserving if $\mathcal B = 0$ for each interface boundary.
\label{thm:BoundedEntropyThm}
\end{thm}
\begin{proof}
When $\mathcal B\ge 0$, the integrals along $\Gamma_b$ and $\Gamma_c$ in \eqref{eq:totalEntropy2} are non-negative so
\begin{equation}
\frac{d \bar S}{dt} \le -\left\{\int_{\Gamma_a}  \spacevec f^\epsilon(\statevec u) \cdot\hat n_a\dS  + \int_{\Gamma_d} \spacevec f^\epsilon(\statevec v) \cdot\hat n_d\dS \right\}.
\label{eq:MultiDEntropyBound1}
\end{equation}
The time derivative of the total entropy \eqref{eq:totalEntropy2} can therefore be bounded solely in terms of the physical boundary values and the initial condition, from which the result follows. When  $\mathcal B= 0$, the inequality in \eqref{eq:MultiDEntropyBound1} becomes an equality, and any entropy gained or lost is done through the physical boundaries.
\end{proof}

\subsection{Conservation Conditions}

We get the conservation conditions on the penalty functions when we integrate \eqref{eq:TwoEquationStrongForm} over the domains, apply  Gauss' law to each flux divergence integral, and combine the results as in \eqref{eq:OversetNorm} (so as to not double count the overlap region), to get
\begin{equation}
\begin{split}
\int_{\Omega_u} {\partial_t \statevec u}d\Omega_u &+ \int_{\Omega_v} {\partial_t \statevec v}d\Omega_v -  \left\{\eta\int_{\Omega_O} {\partial_t \statevec u}d\Omega_O +(1-\eta)\int_{\Omega_O} {\partial_t \statevec v}d\Omega_O\right\}
\\& + \int_{\Gamma_b}\left(\mathcal P_u^b + \mathcal P^b_v\right)\dS +  \int_{\Gamma_c}\left(\mathcal P_u^c + \mathcal P^c_v\right)\ \dS 
\\&+ 
\int_{\Gamma_a}\bigstatevec f(\statevec u)\cdot\hat n_a \dS            + \int_{\Gamma_c}\bigstatevec f(\statevec u)\cdot\hat n_c  \dS
\\&+ \int_{\Gamma_d}\bigstatevec f(\statevec v)\cdot\hat n_d \dS   + \int_{\Gamma_b}\bigstatevec f(\statevec v)\cdot\hat n_b \dS
\\&- \int_{\Gamma_b}\left[\eta\bigstatevec f(\statevec u) + (1-\eta)\bigstatevec f(\statevec v) \right]\cdot\hat n_b \dS  
\\&- \int_{\Gamma_c}\left[\eta\bigstatevec f(\statevec u) + (1-\eta)\bigstatevec f(\statevec v) \right]\cdot\hat n_c  \dS
\\&= 0.
\end{split}
\label{eq:BigMultiDWeakForm3}
\end{equation}
Gathering terms,
\begin{equation}
\begin{split}
\frac{d}{dt}\left\{\int_{\Omega_u} { \statevec u}d\Omega_u\right. &+ \left.\int_{\Omega_v} { \statevec v}d\Omega_v - \left\{ \eta\int_{\Omega_O} { \statevec u}d\Omega_O + (1-\eta)\int_{\Omega_O} { \statevec v}d\Omega_O\right\}\right\}\\& + \int_{\Gamma_b}\left[\mathcal P_u^b + \mathcal P^b_v - \eta\left( \bigstatevec f(\statevec u) -\bigstatevec f(\statevec v)\right)\cdot\hat n_b \right]\dS 
\\& +  \int_{\Gamma_c}\left[\mathcal P_u^b + \mathcal P^b_v + (1-\eta)\left( \bigstatevec f(\statevec u) -\bigstatevec f(\statevec v)\right)\cdot\hat n_c \right]\ \dS 
\\&+ 
\int_{\Gamma_a}\bigstatevec f(\statevec u)\cdot\hat n_a \dS   
+ \int_{\Gamma_d}\bigstatevec f(\statevec v)\cdot\hat n_d \dS 
\\&= 0.
\end{split}
\label{eq:BigMultiDConservation}
\end{equation}

The formulation is conservative when the terms along the artificial interface boundaries vanish, leaving only the physical boundary fluxes, i.e., when
\begin{equation}
\left( \bigstatevec f(\statevec u) -\bigstatevec f(\statevec v)\right)\cdot\spacevec\beta + \mathcal P_u + \mathcal P_v  = 0,
\label{eq:MultiDConservationCondition}
\end{equation}
where $\spacevec \beta$ is defined as before in \eqref{eq:BetaCondition2D}, holds along each of $\Gamma_b$ and $\Gamma_c$.

When we define an integral quantity over a domain $V$, $\bar{\statevec q}_{V} = \int_{V} \statevec q dV$, and assume that \eqref{eq:MultiDConservationCondition} holds on each artificial interface boundary,  the statement of conservation becomes
\begin{equation}
\frac{d}{dt}\left\{ \bar {\statevec u}_{\Omega_u} + \bar {\statevec v}_{\Omega_v}- \left\{\eta \bar {\statevec u}_{\Omega_O}+ (1-\eta) \bar {\statevec v}_{\Omega_O}\right\}\right\} = -\left\{\int_{\Gamma_a}\bigstatevec f(\statevec u)\cdot\hat n_a \dS   
+ \int_{\Gamma_d}\bigstatevec f(\statevec v)\cdot\hat n_d \dS \right\}.
\label{eq:Conservation1}
\end{equation}
When $ {\statevec u} = {\statevec v} =  \statevecGreek\omega$, \eqref{eq:Conservation1} becomes the usual conservation statement over $\Omega$,
\begin{equation}
\frac{d}{dt}\bar{\statevecGreek\omega}_\Omega = -\left\{\int_{\Gamma_a}\bigstatevec f(\statevecGreek\omega)\cdot\hat n_a \dS   
+ \int_{\Gamma_d}\bigstatevec f(\statevecGreek\omega)\cdot\hat n_d \dS \right\}.
\end{equation}
\begin{rem}
\textcolor{black}{Since $\mathcal P_u = 0$ and  $\mathcal P_v=0$ when $ {\statevec u} = {\statevec v}$, } the conservation condition \eqref{eq:MultiDConservationCondition} is trivially satisfied when $ {\statevec u} = {\statevec v}$.
Conversely, the overset problem is not conservative without a penalty unless $\statevec f(\statevec u)=\statevec f(\statevec v)$.
\end{rem}

\subsection{Finding the Entropy Bounding and Conserving Penalty Functions}
We find the penalty functions from the two equations \eqref{eq:MultiDEntropyCondition} and \eqref{eq:MultiDConservationCondition}, copied here,
\begin{equation}
\left( \spacevec f^\epsilon(\statevec u)  -  \spacevec f^\epsilon(\statevec v)\right)\cdot\spacevec \beta+\statevec w_u^T\mathcal P_u + \statevec w_v^T\mathcal P_v\ge 0\quad \text{(Entropy Boundedness)},
\label{eq:MultiDEntropyConditionCopy}
\end{equation}
\begin{equation}
\left( \bigstatevec f(\statevec u) -\bigstatevec f(\statevec v)\right)\cdot\spacevec\beta + \mathcal P_u + \mathcal P_v  = 0. \quad \text{(Conservation)}
\label{eq:MultiDConservationConditionCopy}
\end{equation}
The equations are linear in $\mathcal P_u $ and $\mathcal P_v$.

From the conservation condition,
\begin{equation}
\mathcal P_v = -\left(\spacevec\beta\cdot\left(\bigstatevec f(\statevec u)-\bigstatevec f(\statevec v) \right)   +  \mathcal P_u\right).
\label{eq:consConditionForSigmav}
\end{equation}
Substituting \eqref{eq:consConditionForSigmav} into \eqref{eq:MultiDEntropyConditionCopy},
\begin{equation}
\spacevec\beta\cdot\left(\spacevec f^\epsilon(\statevec u) - \spacevec f^\epsilon(\statevec v)\right) + \statevec w^T_u \mathcal P_u -\statevec w^T_v\left(\spacevec\beta\cdot\left(\bigstatevec f(\statevec u)-\bigstatevec f(\statevec v) \right)   +  \mathcal P_u\right)\ge 0.
\end{equation}
Re-arranging gives an equation for $\mathcal P_u$,
\begin{equation}
\spacevec\beta\cdot\left( \spacevec f^\epsilon(\statevec u) - \spacevec f^\epsilon(\statevec v)- \statevec w_v^T\left(\bigstatevec f(\statevec u)-\bigstatevec f(\statevec v) \right) \right)  +\left(\statevec w_u -\statevec w_v\right)^T\mathcal P_u\ge 0.
\end{equation}
If, instead, we write
\begin{equation}
\mathcal P_u = -\left(\spacevec\beta\cdot\left(\bigstatevec f(\statevec u)-\bigstatevec f(\statevec v) \right)   +  \mathcal P_v\right),
\label{eq:consConditionForSigmau}
\end{equation}
then we get an equation for $\mathcal P_v$,
\begin{equation}
\spacevec\beta\cdot\left(\spacevec f^\epsilon(\statevec u) - \spacevec f^\epsilon(\statevec v)- \statevec w_u^T\left(\bigstatevec f(\statevec u)-\bigstatevec f(\statevec v) \right) \right)
-\left(\statevec w_u -\statevec w_v\right)^T\mathcal P_v\ge 0.
\end{equation}
Gathering the two, we have equations for the penalty functions, 
\begin{equation}
\begin{gathered}
\spacevec\beta\cdot\left(\spacevec f^\epsilon(\statevec u) - \spacevec f^\epsilon(\statevec v)- \statevec w_v^T\left(\bigstatevec f(\statevec u)-\bigstatevec f(\statevec v) \right) \right)  +\left(\statevec w_u -\statevec w_v\right)^T\mathcal P_u\ge 0\hfill\\
\spacevec\beta\cdot\left(\spacevec f^\epsilon(\statevec u) - \spacevec f^\epsilon(\statevec v)- \statevec w_u^T\left(\bigstatevec f(\statevec u)-\bigstatevec f(\statevec v) \right) \right)
 -\left(\statevec w_u -\statevec w_v\right)^T\mathcal P_v\ge 0,\hfill
\end{gathered}
\label{eq:sigmaMatricesConditions}
\end{equation}
which must hold for any $\spacevec \beta = \beta\hat n$. Taking $\mathcal P = \spacevec {\mathcal P}\cdot\hat n$, it is sufficient for each component to satisfy
\begin{equation}
\begin{gathered}
\beta\left( f_i^\epsilon(\statevec u) -  f_i^\epsilon(\statevec v)- \statevec w_v^T\left(\statevec f_i(\statevec u)-\statevec f_i(\statevec v) \right) \right)  +\left(\statevec w_u -\statevec w_v\right)^T\mathcal P_{i,u}\ge 0\hfill\\
\beta\left( f_i^\epsilon(\statevec u) -  f_i^\epsilon(\statevec v)- \statevec w_u^T\left(\statevec f_i(\statevec u)-\statevec f_i(\statevec v) \right) \right)
 -\left(\statevec w_u -\statevec w_v\right)^T\mathcal P_{i,v}\ge 0.\hfill
\end{gathered}
\label{eq:ComponentConditions}
\end{equation}

It is convenient to re-write \eqref{eq:ComponentConditions} in terms of the standard jump operator, $\jump{\cdot}$. Let us drop the subscript $i$ in the following and define $\jump{f^\epsilon} \equiv f^\epsilon(\statevec u) - f^\epsilon(\statevec v)$, etc. Then, compactly,
\begin{equation}
\begin{gathered}
\beta\left( \jump{f^\epsilon} - \statevec w^T_v\jump{\statevec f}\right)+ \jump{\statevec w}^T\mathcal P_u\ge 0\hfill\\
\beta\left( \jump{f^\epsilon} - \statevec w^T_u\jump{\statevec f}\right)- \jump{\statevec w}^T\mathcal P_v\ge 0\hfill
\end{gathered}
\label{eq:jumpVersionsOfPenaltyInequality}
\end{equation}
for each component in each coordinate direction, $x_i$.

Since \textcolor{black}{$\mathcal P_u$ and $\mathcal P_v$} appear in \eqref{eq:jumpVersionsOfPenaltyInequality} premultiplied by a vector, getting an analytic expression for each penalty function requires us to be able to write the remaining terms of the form $\left( \jump{f^\epsilon} - \statevec w^T\jump{\statevec f}\right)$ as some other vector premultiplied by $\jump{\statevec w}^T$. At the very least, we need to be able to factor the jump in the entropy flux, $ \jump{f^\epsilon}$, into the product of $\jump{\statevec w}^T$ and some vector.

The required factorization of the jump in the entropy flux can be found in the so-called Tadmor jump condition \cite{Tadmor:2003fv},
\begin{equation}
\jump{f^\epsilon} = \jump{\statevec w^T\statevec f} - \jump{\statevec w}^T\statevec f^{ec}.
\label{eq:TadmorJump0}
\end{equation}
The flux, $\statevec f^{ec}(\statevec u,\statevec v)$, called the ``entropy conserving flux", is generically defined in terms of a path integral in phase space. It is consistent with the conservative flux, $\statevec f$, meaning that when the arguments are equal,  $\statevec f^{ec}(\statevec u,\statevec u) = \statevec f(\statevec u)$. Closed form solutions have been derived for a variety of important systems of equations like the shallow water, Euler gas-dynamics, and MHD equations. See \ref{A:TadmorJump} for further discussion.

Starting from the first equation of \eqref{eq:jumpVersionsOfPenaltyInequality} with the equality, we substitute $\jump{f^\epsilon}$ using  \eqref{eq:TadmorJump0},
\begin{equation}
\beta\left( \jump{\statevec w^T\statevec f} - \jump{\statevec w}^T\statevec f^{ec}- \statevec w^T_v\jump{\statevec f}\right)+ \jump{\statevec w}^T\mathcal P_u= 0.
\end{equation}
Now,
\begin{equation}
\jump{\statevec w^T\statevec f} = \statevec w^T_u\statevec f(\statevec u) -  \statevec w^T_v\statevec f(\statevec v),
\end{equation}
so
\begin{equation}
\begin{split}
\jump{\statevec w^T\statevec f} - \statevec w^T_v\jump{\statevec f} &=  \statevec w^T_u\statevec f(\statevec u) -  \statevec w^T_v\statevec f(\statevec v) - \statevec w^T_v( \statevec f(\statevec u) - \statevec f(\statevec v))
\\&
=  \statevec w^T_u\statevec f(\statevec u)  -  \statevec w^T_v\statevec f(\statevec u) \\&= \jump{\statevec w}^T\statevec f(\statevec u),
\end{split}
\end{equation}
Replacing that gives
\begin{equation}
\beta\jump{\statevec w}^T\left( \statevec f(\statevec u)- \statevec f^{ec}  \right) +  \jump{\statevec w}^T\mathcal P_u= 0.
\end{equation}
We can then factor out the jump, and
\begin{equation}
\jump{\statevec w}^T\left(\beta\left( \statevec f(\statevec u)- \statevec f^{ec}  \right) +\mathcal P_u\ \right) = 0\quad \forall \jump{\statevec w}.
\end{equation}
Therefore,
\begin{equation}
\beta\left( \statevec f(\statevec u)- \statevec f^{ec}  \right) +\mathcal P_u = 0,
\end{equation}
which means that each directional component $ \mathcal P_u$ is a state vector
\begin{equation}
\mathcal P_u(\statevec u,\statevec v) = \beta\left( \statevec  f^{ec} - \statevec f(\statevec u)  \right).
\end{equation}
Similarly, 
\begin{equation}
\beta\left( \jump{f^\epsilon} - \statevec w^T_u\jump{\statevec f}\right)- \jump{\statevec w}^T\mathcal P_v= 0.
\end{equation}
This time,
\begin{equation}
\begin{split}
\jump{\statevec w^T\statevec f} - \statevec w^T_u\jump{\statevec f} &=  \statevec w^T_u\statevec f(\statevec u) -  \statevec w^T_v\statevec f(\statevec v) - \statevec w^T_u( \statevec f(\statevec u) - \statevec f(\statevec v))
\\&
=  \statevec w^T_u\statevec f(\statevec v)  -  \statevec w^T_v\statevec f(\statevec v) = \jump{\statevec w}^T\statevec f(\statevec v),
\end{split}
\end{equation}
from which it follows that
\begin{equation}
\mathcal P_v(\statevec u, \statevec v) = -\beta\left( \statevec  f^{ec} - \statevec f(\statevec v)  \right).
\end{equation}
Note that by consistency of the entropy conserving flux, $\mathcal P_v(\statevec v,\statevec v) = \beta(\statevec f(\statevec v) - \statevec f(\statevec v)) = 0$, and similarly for $\mathcal P_u$; the penalties vanish when the two solutions are the same, as required.

\subsection{Summary}\label{Sec:Summary}
We summarize the results of this section in the following theorem:
\begin{thm}
The overset domain problem, where
\begin{equation}
\begin{gathered}
(O)\quad \partial_t \statevec u + \vecNabla_x \cdot \bigstatevec f(\statevec u) + \mathcal L^b\left[\spacevec{\mathcal P}_{u}(\statevec u, \statevec v)\cdot\hat n_b\right] + \mathcal L^c\left[\spacevec{\mathcal P}_{u}(\statevec u, \statevec v)\cdot\hat n_c\right]= 0\quad \vec x\in \Omega_u\hfill\\
(B)\quad \partial_t \statevec v + \vecNabla_x \cdot \bigstatevec f(\statevec v) + \mathcal L^b\left[\spacevec{\mathcal P}_{v}(\statevec u, \statevec v)\cdot\hat n_b\right] + \mathcal L^c\left[\spacevec{\mathcal P}_{v}(\statevec u, \statevec v)\cdot\hat n_c\right]= 0\quad \vec x\in \Omega_v,\hfill
\end{gathered}
\end{equation}
is entropy conserving and conservative if
\begin{equation}
 \mathcal P_{i,u}(\statevec u, \statevec v) = \beta\left( \statevec  f_i^{ec}(\statevec u,\statevec v) - \statevec f_i(\statevec u)  \right)
 \label{eq:PuFinal}
\end{equation}
\begin{equation}
\mathcal P_{i,v}(\statevec u, \statevec v) = -\beta\left( \statevec  f_i^{ec}(\statevec u,\statevec v) - \statevec f_i(\statevec v)  \right),
\label{eq:PvFinal}
\end{equation}
where $\statevec f_i^{ec}(\statevec u,\statevec v)$ is an entropy conserving flux for the $i^{th}$ coordinate direction that satisfies the Tadmor jump condition, \eqref{eq:TadmorJump}.
\end{thm}
\begin{proof}
We first show conservation. When we multiply each component in \eqref{eq:PuFinal} and \eqref{eq:PvFinal} by a component of a unit vector along any of the interface boundaries and add the results,
\begin{equation}
\mathcal P_u + \mathcal P_v = -\spacevec\beta\cdot(\bigstatevec f(\statevec u) - \bigstatevec f(\statevec v)),
\end{equation}
which is the conservation condition, \eqref{eq:MultiDConservationCondition}/\eqref{eq:MultiDConservationConditionCopy}.

To show that entropy is preserved, we show \eqref{eq:MultiDEntropyConditionCopy} with the equality, using the Tadmor jump condition, \eqref{eq:TadmorJump}. For each component,
\begin{equation}
\begin{split}
&\beta\left(f_i^\epsilon(\statevec u) - f_i^\epsilon(\statevec v)\right) + \statevec w^T_u \mathcal P_{i,u} +\statevec w^T_v \mathcal P_{i,v} 
\\&=\beta\left\{\jump{f_i^\epsilon} + \statevec w^T_u \left( \statevec  f_i^{ec}(\statevec u,\statevec v) - \statevec f_i(\statevec u)  \right) -\statevec w^T_v \left( \statevec  f_i^{ec}(\statevec u,\statevec v) - \statevec f_i(\statevec v)  \right) \right\}
\\& =\beta\left\{\jump{f_i^\epsilon} +\jump{\statevec w}^T\statevec f_i^{ec} - \jump{\statevec w^T\statevec f_i}   \right\}
\\& = 0.
\end{split}
\end{equation}
Then \eqref{eq:MultiDEntropyCondition}/\eqref{eq:MultiDEntropyConditionCopy} holds for any $\spacevec \beta = \beta\hat n$, and by Thm. \ref{thm:BoundedEntropyThm}, the formulation is entropy conserving.
\end{proof}

The penalty terms are therefore proportional to the difference between an entropy conserving flux for the specific system of equations and the conservative flux, rather than being directly proportional to the difference in the states, as was the case for linear problems considered in Paper I \cite{KOPRIVA2022110732}. In the case of the Burgers equation, however, we show in the following example that the penalty terms can be written that way, but with a solution-dependent coefficient that depends on the two states.

\subsection{Example}
The Burgers equation in one space dimension is
\begin{equation}
\textcolor{black}{\partial_t\omega}+ \partial_x\left(\frac{1}{2}\omega^2\right)=0.
\end{equation}
We pose the problem on the one-dimensional domains shown in Fig. \ref{fig:OSet1D}. 
\begin{figure}[tbp] 
   \centering
   \includegraphics[width=4in]{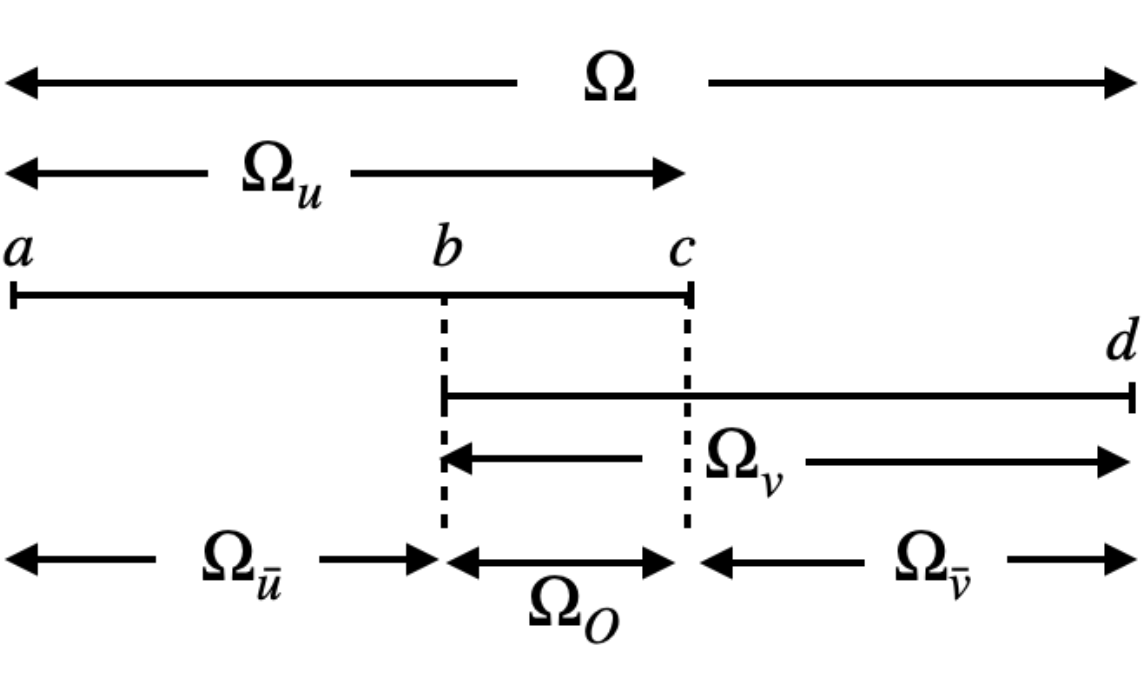} 
   \caption{The overset domain problem in one space dimension}
   \label{fig:OSet1D}
\end{figure}

The Burgers equation has an entropy pair $\left(s,f^\epsilon\right) = \left(\oneHalf\omega^2,\frac{1}{3}\omega^3\right)$, conservative flux $f = \frac{1}{2}\omega^2$, and entropy variable $w = \partial s/\partial \omega = \omega$. 
The symmetric and consistent entropy conserving flux for this entropy that satisfies the Tadmor jump condition is
\begin{equation}
f^{ec}(u,v) = \frac{1}{6}\left( u^2 + uv + v^2\right),
\end{equation}
which can be verified by substitution.

For the one-dimensional geometry of Fig. \ref{fig:OSet1D}, $\hat n_b = -\hat x$ and  $\hat n_c = \hat x$, so
\begin{equation}
\beta = \left\{
\begin{gathered}
\eta, \quad x=b \hfill\\ 
(1-\eta),\quad x=c.
\end{gathered}
   \right.
\label{eq:BetaCondition1D}
\end{equation}
Then the penalty term $\mathcal P_u$ is
\begin{equation}
\begin{split}
\mathcal P_u(u,v) &= \beta\left( f^{ec}(u,v) - f(u)\right)
\\&= \frac{\beta}{6}\left( u^2 + uv + v^2\right) - \beta\oneHalf u^2 
\\& =  \frac{\beta}{6}\left( -2u^2 + uv + v^2\right)
\\& = - \frac{\beta}{6}\left( \left(u^2 - v^2 \right) + u(u-v) \right)
\\& = - \frac{\beta}{6}(2u+v)(u-v).
\end{split}
\label{eq:PB1}
\end{equation}
Explicitly, the penalty vanishes when the jump in the solutions is zero, as required.

Similarly,
\begin{equation}
\begin{split}
\mathcal P_v(u,v) &= -\frac{\beta}{6}\left( u^2 + uv + v^2\right) +\beta \oneHalf v^2 
\\& = - \frac{\beta}{6}\left( u^2 + uv -2 v^2\right)
\\& = - \frac{\beta}{6}\left( \left(u^2 - v^2 \right) + v(u-v) \right)
\\& = - \frac{\beta}{6}(u+2v)(u-v).
\end{split}
\label{eq:PB2}
\end{equation}
The Burgers equation penalties can be written, then, in the form $\sigma(u,v)(u-v)$, and is linear in the jump.

\begin{rem} 
For the scalar Burgers equation, one can get the same penalty functions directly from \eqref{eq:jumpVersionsOfPenaltyInequality}, without the explicit need for the jump condition or entropy conserving flux, since it is easy to factor the jump directly using standard polynomial factorization formulas.
\end{rem}

The entropy preserving and conservative overset domain equations for the Burgers equation can therefore be written as
\begin{equation}
\begin{gathered}
 \partial_t  u + \partial_x  f( u) - \mathcal L^b_u\left[ \frac{\eta}{6}(2u+v)(u-v) \right] -
 \mathcal L^c_u\left[ \frac{1-\eta}{6}(2u+v)(u-v)\right] = 0,\quad x\in\Omega_u\hfill\\
  \partial_t  v + \partial_x  f( v) - \mathcal L^b_v\left[  \frac{\eta}{6}(u+2v)(u-v) \right] -
 \mathcal L^c_u\left[  \frac{1-\eta}{6}(u+2v)(u-v) \right] = 0,\quad x\in\Omega_v\hfill
 \end{gathered}
 \label{eq:OversetDomainBurgers1D}
\end{equation}
for $0<\eta<1$.

In one space dimension, the lifting operator reduces to a point,
\begin{equation}
    \int_I \statevecGreek\psi^T\mathcal L^\alpha[\statevecGreek\phi]dx = \left.\statevecGreek\psi^T\statevecGreek\phi\right|_\alpha, 
\end{equation}
so in weak form,
\begin{equation}
\begin{gathered}
\iprod{\phi_u, \partial_t  u}_{\Omega_u} + \iprod{\phi_u,\partial_x  f( u)} _{\Omega_u}
- \left[\phi_u \frac{\eta}{6}(2u+v)(u-v) \right]_b -
\left[\phi_u \frac{1-\eta}{6}(2u+v)(u-v)\right] _c = 0\hfill\\
\iprod{\phi_v,  \partial_t  v}_{\Omega_v} + \iprod{\phi_v,\partial_x  f( v)}_{\Omega_v} 
-  \left[ \phi_v \frac{\eta}{6}(u+2v)(u-v) \right]_b -
\left[\phi_v\frac{1-\eta}{6}(u+2v)(u-v)\right]_c = 0.\hfill\\
 \end{gathered}
 \label{eq:GeneralWeakForm}
\end{equation}

\section{Interior Coupling and Dissipation}

One can optionally couple the solutions further and add dissipation in the overlap region by adding internal penalty terms \textcolor{black}{as an entropy bounded way to enforce volume coupling}.  Unlike the interface penalties, the linear analysis of Paper I applies to the interior penalties for nonlinear problems. The only difference is that the penalties are now proportional to the entropy variables rather than the state variables. For that reason we include only an outline here.

We add linear penalties at any set of points, $\{\spacevec x^m\}_{m=1}^M$, located anywhere in $\Omega_O$, to the overset problem \eqref{eq:TwoEquationStrongForm}, the additions being
\begin{equation}
\begin{gathered}
\frac{1}{M}\sum_{m=1}^M\mathcal L_u^m \left[ \mmatrix{$\Sigma$}_u^m (\statevec w_u-\statevec w_v)\right],\quad\spacevec x\in\Omega_u \hfill\\
  \frac{1}{M}\sum_{m=1}^M\mathcal L_v^m \left[ \mmatrix{$\Sigma$}_v^m (\statevec w_v-\statevec w_u)\right],\quad\spacevec x\in\Omega_v,\hfill
 \end{gathered}
 \label{eq:OversetDomainProblem1DPwInt}
\end{equation}
where $\mmatrix{$\Sigma$}_{u,v}>0$ are positive definite matrices, which can be constant, and $\mathcal L_{u,v}^m$ are lifting operators that select the value at $\spacevec x^m$.

 \textcolor{black}{Then the strong form, \eqref{eq:TwoEquationStrongForm}, of the system is amended to
\begin{equation}
\begin{gathered}
(O)\quad \partial_t \statevec u + \vecNabla_x \cdot \bigstatevec f(\statevec u) + \mathcal L^b\left[\mathcal P^b_{u}(\statevec u, \statevec v)\right] + \mathcal L^c\left[\mathcal P^c_{u}(\statevec u, \statevec v)\right] + \frac{1}{M}\sum_{m=1}^M\mathcal L_u^m \left[ \mmatrix{$\Sigma$}_u^m (\statevec w_u-\statevec w_v)\right]= 0\quad \vec x\in \Omega_u\hfill\\
(B)\quad \partial_t \statevec v + \vecNabla_x \cdot \bigstatevec f(\statevec v) + \mathcal L^b\left[\mathcal P^b_{v}(\statevec u, \statevec v)\right] + \mathcal L^c\left[\mathcal P^c_{v}(\statevec u, \statevec v)\right] +  \frac{1}{M}\sum_{m=1}^M\mathcal L_v^m \left[ \mmatrix{$\Sigma$}_v^m (\statevec w_v-\statevec w_u)\right]= 0\quad \vec x\in \Omega_v,\hfill
\end{gathered}
\label{eq:TwoEquationStrongFormWithInteriorPenalty}
\end{equation}
while the weak forms become
\begin{equation}
\begin{gathered}
\begin{split}
(O)\quad \iprod{\statevecGreek\phi_u,\partial_t \statevec u} +\iprod{\statevecGreek\phi_u,\vecNabla_x \cdot \bigstatevec f(\statevec u)} 
&+ \int_{\Gamma_b}\statevecGreek\phi_u^T\mathcal P^b_{u}(\statevec u, \statevec v)\dS
+ \int_{\Gamma_c}\statevecGreek\phi_u^T\mathcal P^c_{u}(\statevec u, \statevec v)\dS \\&+ \frac{1}{M}\sum_{m=1}^M \left[ \mmatrix{$\Sigma$}_u^m (\statevec w_u-\statevec w_v)\right]_{\spacevec x_m}= 0\quad \vec x\in \Omega_u
\end{split}\hfill\\
\begin{split}
(B)\quad  \iprod{\statevecGreek\phi_v,\partial_t \statevec v} + \iprod{\statevecGreek\phi_v,\vecNabla_x \cdot \bigstatevec f(\statevec v)} &+ \int_{\Gamma_b}\statevecGreek\phi_v^T\mathcal P^b_{v}(\statevec u, \statevec v)\dS + \int_{\Gamma_c}\statevecGreek\phi_v^T\mathcal P^c_{v}(\statevec u, \statevec v)\dS\\&+  \frac{1}{M}\sum_{m=1}^M \left[ \mmatrix{$\Sigma$}_v^m (\statevec w_v-\statevec w_u)\right]_{\spacevec x_m}= 0\quad \vec x\in \Omega_v.
\end{split}\hfill
\end{gathered}
\label{eq:TwoEquationWeakFormWithInteriorPenalty}
\end{equation}
}

\textcolor{black}{Getting the entropy bound then follows the procedure used in Sec. \ref{EntropyBoundednessSection}.
When one replaces the test functions $\statevecGreek\phi$ with the entropy variables, and uses  the interior penalty terms derived in Sec. \ref{Sec:Summary}, the entropy equation becomes
\begin{equation}
\frac{d \bar S}{dt} + \mathcal D\le -\left\{\int_{\Gamma_a}  \spacevec f^\epsilon(\statevec u) \cdot\hat n_a\dS  + \int_{\Gamma_d} \spacevec f^\epsilon(\statevec v) \cdot\hat n_d\dS \right\},
\label{eq:MultiDEntropyBound2}
\end{equation}
where}
\begin{equation}
\mathcal D \equiv  \frac{1}{M}\sum_{m=1}^M \mathcal P^m= \frac{1}{M}\sum_{m=1}^M\left\{(1-\eta) \statevec w_u^T\mmatrix{$\Sigma$}_u^m (\statevec w_u-\statevec w_v) + \eta\statevec w_v^T\mmatrix{$\Sigma$}_v^m (\statevec w_v-\statevec w_u)\right\}_{\vec x^m}.
\end{equation}
\textcolor{black}{To ensure that the total entropy is bounded, then, we require that each term in $\mathcal D$ is non-negative, that is}, $\mathcal P^m$ satisfies
\begin{equation}
\mathcal P^m = (1-\eta) \statevec w_u^T\mmatrix{$\Sigma$}_u^m (\statevec w_u-\statevec w_v) + \eta\statevec w_v^T\mmatrix{$\Sigma$}_v^m (\statevec w_v-\statevec w_u) \ge 0.
\label{eq:InteriorPenaltyCondition}
\end{equation}
Eq. \eqref{eq:InteriorPenaltyCondition} is identical in form to the linear case, with entropy variables replacing state variables. From the analysis of Sec. 3.2 of Paper I, \textcolor{black}{the condition on the penalty matrices is}
\begin{equation}
(1-\eta)\mmatrix{$\Sigma$}^m_u = \eta\mmatrix{$\Sigma$}_v^m.
\label{eq:InteriorPenaltyCoupling}
\end{equation}
With this condition holding,
\begin{equation}
\mathcal P^m = (1-\eta)(\statevec w_u - \statevec w_v)^T\mmatrix{$\Sigma$}^m_u (\statevec w_u - \statevec w_v)\ge 0,
\end{equation}
provided that the penalty matrices are chosen to be  positive definite. \textcolor{black}{Then \eqref{eq:MultiDEntropyBound1} follows from \eqref{eq:MultiDEntropyBound2} and the formulation is entropy conserving if all $\mathcal P^m = 0$, and entropy bounded otherwise.} The relationship \eqref{eq:InteriorPenaltyCoupling} between the penalty matrices also ensures conservation \cite{KOPRIVA2022110732}.

\begin{rem}
This section shows how to add dissipation to the results of the previous section by adding such penalties at the interface points alone. This is equivalent to the creation of entropy dissipative fluxes from entropy conserving ones. See \cite{Winters2021}, or \cite{Ranocha2017},\cite{Winters2016} for examples.
\end{rem}

\section{Discussion and Conclusion}

\textcolor{black}{We have derived entropy preserving and entropy dissipative formulations of the overset domain problem for nonlinear hyperbolic conservation laws. All that is necessary is that the system has a convex entropy and entropy flux associated with it. As for linear problems studied in Paper I, two-way coupling is necessary, and in this formulation the two-way coupling is enforced by penalty terms in both domains along the artificial interior boundaries. With the required two-way interface coupling and the optional volume coupling, the overset domain problems are provably conservative. Finally, overset domain equations satisfy the same global entropy bound and are consistent with the original single domain problem they represent. }

\textcolor{black}{Unlike the penalties for linear problems, the penalties for nonlinear problems are not necessarily linear in the state or entropy variables. The interface penalties were shown to depend conveniently on the difference between the conservative flux and a two-point entropy flux that satisfies the Tadmor jump condition. Closed form versions of these two-point fluxes exist for numerous systems of equations, including the Euler gas-dynamics equations, making the penalties straightforward to express for those systems. Otherwise, a representation of these fluxes as integrals in phase space always exists \cite{Tadmor:2003fv}. Entropy dissipation and additional overlap coupling between the domains can be added through penalties that are linear, but in terms of the difference between the entropy variables from the overset domains rather than the state variables. }

\textcolor{black}{The overlapping domain equations serve to match the solutions of the original PDE system on the original domain sans overlaps, and are independent of the methods used to approximate them. Operationally, strong form approximations such as finite difference methods would start by approximating  \eqref{eq:TwoEquationStrongFormWithInteriorPenalty} in the most general form. Such approximations would require consistent interpolation operators on the finite difference grids and discrete lifting operators to evaluate them along the interior interfaces. Finite or spectral element methods would start from \eqref{eq:TwoEquationWeakFormWithInteriorPenalty}. Interpolation is well defined in those approximations, but the integrals along the interior boundary curves would need to be approximated in a stable way. Ultimately, the design of a stable numerical overset grid method must satisfy the same entropy bound as shown for the PDEs, and those that don't cannot be stable. As far as we know, there are no overset grid approximations today in more than one space dimension that approximate boundary value problems for hyperbolic systems that are well-posed or entropy bounded like those presented here. }

\textcolor{black}{The new entropy conservative/bounded overset domain problems are flexible in that they include parameters that can be chosen or adapted as desired. The domain weighting parameter, $0<\eta<1$, for instance, can be chosen to favor one grid over the other, which may be useful if one grid is known to have a more accurate solution. Entropy dissipation can be added and controlled through linear penalty terms added at the artificial interface boundaries (as is done with entropy stable Riemann solvers) or at arbitrary points in the interior of the overlap region. This gives the flexibility to tune or adapt, for instance, the amount of entropy dissipation in the PDE system according to the solutions through the size of the penalty matrices, as desired.
}
\appendix
\section{The Tadmor Jump Condition}\label{A:TadmorJump}

For completeness, we review the derivation of the Tadmor jump condition \cite{Tadmor:2003fv} for an entropy conserving numerical flux function, $\statevec f^{ec}$. We start with the one dimensional hyperbolic system with periodic boundary conditions
\begin{equation}
\partial_t\statevecGreek\omega + \partial_x\,\statevec{f}(\statevecGreek\omega) = 0,
\end{equation}
equipped with an entropy pair $(s,f^\epsilon)$ . We then write its conservative finite difference approximation in space
\begin{equation}
\frac{\partial}{\partial t} {\statevecGreek\omega}_j + \frac{\statevec{f}^*_{j+1/2} - \statevec{f}^*_{j-1/2}}{\Delta x} = 0,\quad j\in\mathbb{Z},
\end{equation}
where the numerical flux, $\statevec{f}^*_{j+1/2} = \statevec{f}^*_{j+1/2}( {\statevecGreek\omega}_{j+1}, {\statevecGreek\omega}_j)$, is the continuous flux with the consistency property $\statevec{f}^*_{j+1/2}(\statevecGreek\omega,\statevecGreek\omega) = \statevec{f}(\statevecGreek\omega)$. The quantity $\Delta x$ is the width of the grid cells, and $ {\statevecGreek\omega}_j$ is the discrete value at cell $j$.

Next, we introduce the discrete entropy variable $ {\statevec{w}}_j := \frac{\partial s}{\partial\statevecGreek\omega}( {\statevecGreek\omega}_j)$ and contract the conservative finite difference approximation to mimic the continuous entropy analysis 
\begin{equation}
 {\statevec{w}}_j^T\,\frac{\partial}{\partial t} {\statevecGreek\omega}_j +  {\statevec{w}}_j^T\,\frac{\statevec{f}^*_{j+1/2} - \statevec{f}^*_{j-1/2}}{\Delta x} = 0,\quad j\in\mathbb{Z}.
 \label{eq:A3}
\end{equation}

Assuming time continuity, the first term in \eqref{eq:A3} becomes
\begin{equation}
 {\statevec{w}}_j^T\,\frac{\partial}{\partial t} {\statevecGreek\omega}_j = \frac{\partial}{\partial t}\, {s}_j.
\end{equation}
\textcolor{black}{
The second term in \eqref{eq:A3}, after being multiplied by $\Delta x$, is recast as 
\begin{equation}
\begin{split}
 {\statevec{w}}^T_j\,(\statevec{f}^*_{j+1/2} - \statevec{f}^*_{j-1/2}) &=  (f^{*,\epsilon}_{j+1/2} - f^{*,\epsilon}_{j-1/2} ) - \frac{1}{2}\,(r_{j+1/2} + r_{j-1/2}),
\end{split}
\label{eq:RecastContraction}
\end{equation}
where $f^{*,\epsilon}_{j\pm1/2}$ are consistent numerical entropy fluxes, i.e. ${f}^{*,\epsilon}_{j+1/2} = {f}^{*,\epsilon}_{j+1/2}( {\statevecGreek\omega}_{j+1}, {\statevecGreek\omega}_j)$ and ${f}^{*,\epsilon}_{j+1/2}(\statevecGreek\omega,\statevecGreek\omega) = {f}^{*,\epsilon}(\statevecGreek\omega)$. The last term on the right is an entropy production term, and we get the desired contraction of the fluxes when we find $\statevec{f}^*$ and $f^\epsilon$ so that $r_{j\pm 1/2} = 0$.}

\textcolor{black}{
To find the residual quantities, we define the usual entropy flux potential \cite{Tadmor:2003fv},
\begin{equation}
 {\psi}=  {\statevec{w}}^T\,\statevec{f}(\statevecGreek\omega( {\statevec{w}})) - f^{\epsilon}(\statevecGreek\omega( {\statevec{w}})).
 \label{eq:EntropyFluxPotiential}
\end{equation}
Tadmor writes the numerical entropy flux in terms of this potential,
\begin{equation}
f^{*,\epsilon}_{j+1/2} = \avg{ {\statevec{w}}^T}_{j+1/2}\,\statevec{f}^{*}_{j+1/2} - \avg{ {\psi}}_{j+1/2},
\label{eq:NEFwithPsi}
\end{equation}
where $\avg{\cdot}$ is the average operator, e.g. $\avg{\psi}_{j+1/2} = (\psi_j + \psi_{j+1})/2$. This numerical entropy flux is consistent since it reduces to $f^{\epsilon}$ when the jumps are zero.}

\textcolor{black}{
If we substitute the numerical entropy flux \eqref{eq:NEFwithPsi} into \eqref{eq:RecastContraction}, and rearrange,
\begin{equation}
\begin{split}
\left( \statevec w^T_j - \avg{\statevec w^T}_{j+1/2}\right)\statevec{f}^{*}_{j+1/2} - &\left( \statevec w^T_j - \avg{\statevec w^T}_{j-1/2}\right)\statevec{f}^{*}_{j-1/2} \\&= -\avg{\psi}_{j+1/2} + \avg{\psi}_{j-1/2}  - \frac{1}{2}\,(r_{j+1/2} + r_{j-1/2}).
\end{split}
\end{equation}
If we define the jump operator $\jump{\statevec w}_{j+1/2} \equiv \statevec w_{j+1} - \statevec w_{j}$ and gather terms, then
\begin{equation}
-\oneHalf  \jump{\statevec w^T}_{j+1/2}\statevec{f}^*_{j+1/2} - \oneHalf\jump{\statevec w^T}_{j-1/2}\statevec{f}^*_{j-1/2}+ \avg{\psi}_{j+1/2}  -\avg{\psi}_{j-1/2}=- \frac{1}{2}\,(r_{j+1/2} + r_{j-1/2})
\end{equation}
Now,
\begin{equation}
\begin{split}
\avg{\psi}_{j+1/2}  -\avg{\psi}_{j-1/2} &= \frac{\psi_{j+1}+\psi_j}{2} - \frac{\psi_j +\psi_{j-1}}{2} =\frac{\psi_{j+1}-\psi_j}{2} + \frac{\psi_j -\psi_{j-1}}{2} \\&=\oneHalf \jump{\psi}_{j+1/2} + \oneHalf\jump{\psi}_{j-1/2},
\end{split}
\end{equation}
so
\begin{equation}
\oneHalf\left(-  \jump{\statevec w^T}_{j+1/2}\statevec{f}^*_{j+1/2} +  \jump{\psi}_{j+1/2}\right)+ \oneHalf\left(-\jump{\statevec w^T}_{j-1/2}\statevec{f}^*_{j-1/2}
 + \jump{\psi}_{j-1/2}\right)
=- \frac{1}{2}\,(r_{j+1/2} + r_{j-1/2})
\label{eq:LRMatchEquation}
\end{equation}
We now match the left and right sides of \eqref{eq:LRMatchEquation} to find that
\begin{equation}
r_{j+1/2} = \jump{ {\statevec{w}}^T}_{j+1/2}\,\statevec{f}^*_{j+1/2} - \jump{ {\psi}}_{j+1/2},
\end{equation}
with the equation for $r_{j-1/2}$ having just an integer shift to the left. }

\textcolor{black}{
As mentioned above, one wants the entropy production terms to vanish, $r_{j\pm1/2}=0$, which happens if 
  \begin{equation}
r_{j\pm 1/2}  = 0 = \jump{\statevec{w}}_{j\pm 1/2}^T\statevec{f}^*_{j\pm 1/2} - \jump{\psi}_{j\pm 1/2},
 \end{equation}
i.e.
 \begin{equation}
\jump{\statevec{w}}_{j\pm 1/2}^T\statevec{f}^*_{j\pm 1/2} = \jump{\psi}_{j\pm 1/2}.
\label{eq:ZeroProduction}
 \end{equation}}

\textcolor{black}{
Going back to \eqref{eq:EntropyFluxPotiential}, we can construct the averages
\begin{equation}
\jump{\psi}_{j\pm 1/2} = \jump{\statevec w^T \statevec f}_{j\pm 1/2} - \jump{f^\epsilon}_{j\pm 1/2}.
\end{equation}
Therefore, combining with \eqref{eq:ZeroProduction},
 \begin{equation}
\jump{\statevec{w}}_{j\pm 1/2}^T\statevec{f}^*_{j\pm 1/2} = \jump{\statevec w^T \statevec f}_{j\pm 1/2} - \jump{f^\epsilon}_{j\pm 1/2}.
\label{eq:ap:tadmor}
 \end{equation}
 }

Tadmor \cite{Tadmor:2003fv} defined numerical flux functions $\statevec f^{ec}$ for $\statevec{f}^*$ that satisfy \eqref{eq:ap:tadmor} generally via a path integral in entropy phase space, and so they can always be found, in principle. Although it is relatively straightforward to solve this path integral for the scalar Burgers equation, i.e.
\begin{equation}
\statevec f^{ec}_{j+1/2} = \frac{ {\omega}_j^2 +  {\omega}_j\, {\omega}_{j+1} +  {\omega}^2_{j+1}}{6},
\end{equation}
it has been only recently that closed form analytical expressions of entropy conserving fluxes for complex system of conservations laws such as the shallow water equations \cite{Gassner2016,wintermeyer2017entropy}, the compressible Euler equations \cite{Ismail2009,Ranocha2017}, and the ideal MHD equations \cite{Chandrashekar2016,Winters2016} have been derived.

Once derived, \eqref{eq:ap:tadmor} is a purely algebraic condition dependent only on the entropy variables and potential, independent of any numerical scheme. As such, the entropy conserving flux $\statevec f^{ec}$ satisfies the jump relation for any two arbitrary states and can can be written generally as 
\begin{equation}
\jump{f^\epsilon} = \jump{\statevec w^T\statevec f} - \jump{\statevec w}^T\statevec f^{ec}.
\label{eq:TadmorJump}
\end{equation}

\section*{Acknowledgments}

 This work was supported by a grant from the Simons Foundation (\#426393, David Kopriva).  Gregor Gassner thanks the Klaus-Tschira Stiftung and
the European Research Council for funding through the ERC Starting Grant “An
Exascale aware and Un-crashable Space-Time-Adaptive Discontinuous Spectral
Element Solver for Non-Linear Conservation Laws” (EXTREME, project no. 71448). Jan Nordstr\"om was supported by Vetenskapsrådet, Sweden grant  nr: 2018-05084 VR, 2021-05484 VR  and the Swedish e-Science Research Center (SeRC). 

\bibliographystyle{plain}
\bibliography{OversetNonlinear.bib}

\end{document}

%% file: newcommands.tex
\definecolor{chcol}{rgb}{0.4,0.,0.9}

\newcommand{\halfComma}{\kern 0.083334em}

\newcommand{\avg}[1]{\left\{\hspace*{-3pt}\left\{#1\right\}\hspace*{-3pt}\right\}}
\newcommand{\jump}[1]{\ensuremath{\left\llbracket #1 \right\rrbracket}}
\newcommand{\dS}{{\,\operatorname{dS}}}         
\newcommand\iprod[1]{\left\langle #1\right\rangle}                                             
\newcommand\inorm[1]{\left |\left| #1\right|\right|}                                               
\renewcommand\vec[1]{\accentset{\,\rightarrow}{#1}}                              
\newcommand\spacevec[1]{\accentset{\,\rightarrow}{#1}}                        
\newcommand\statevec[1]{\mathbf #1}                                                     
\newcommand\statevecGreek[1]{\boldsymbol #1}                                     
\newcommand\acclrvec[1]{\accentset{\,\leftrightarrow}{#1}}                      
\newcommand\bigstatevec[1]{\acclrvec{{\mathbf #1}}}                              
\newcommand\vecNabla{\accentset{\,\rightarrow}\nabla}                         
\newcommand\vecNablaX{\accentset{\,\rightarrow}\nabla_{\!x}}              
\newcommand\mmatrix[1]{\underbar{#1}}				

\newcommand\oneHalf{\frac{1}{2}}